\documentclass{amsart}

\usepackage{amsfonts}
\usepackage{amscd}
\usepackage{amssymb}
\usepackage{mathrsfs}
\usepackage{amsmath}
\usepackage{amsthm}
\usepackage{latexsym}
\usepackage[T1]{fontenc}
\usepackage[all]{xy}

\theoremstyle{plain}
\newtheorem{theorem}{Theorem}[section]
\newtheorem{corollary}[theorem]{Corollary}
\newtheorem{lemma}[theorem]{Lemma}
\newtheorem{proposition}[theorem]{Proposition}

\theoremstyle{definition}

\newtheorem{remark}[theorem]{Remark}

\newcommand{\Fix}{{\textup{Fix}}}
\newcommand{\Aper}{{\textup{Aper}}}
\newcommand{\Per}{{\textup{Per}}}

\DeclareMathOperator{\supp}{supp}
\DeclareMathOperator{\id}{id}
\DeclareMathOperator{\AC}{AC}

\numberwithin{equation}{section}

\begin{document}

\title[On the Banach $^*$-algebra associated with a dynamical system]{On the Banach $^*$-algebra crossed product associated with a topological dynamical system}

\author{Marcel de Jeu}
\address{Marcel de Jeu, Mathematical Institute, Leiden University, P.O.\ Box 9512, 2300 RA Leiden, The Netherlands}
\email{mdejeu@math.leidenuniv.nl}

\author{Christian Svensson}
\address{Christian Svensson, Mathematical Institute, Leiden University, P.O.\ Box 9512, 2300 RA Leiden, The Netherlands, and Centre for Mathematical Sciences, Lund University, Box 118, SE-221 00 Lund, Sweden}
\email{svensson.pc@gmail.com}

\author{Jun Tomiyama}
\address{Jun Tomiyama, Department of Mathematics, Tokyo Metropolitan University, Minami-Osawa, Hachioji City, Japan}
\email{juntomi@med.email.ne.jp}


\begin{abstract}
Given a topological dynamical system $\Sigma = (X, \sigma)$, where $X$ is a compact Hausdorff space and $\sigma$ a homeomorphism of $X$, we introduce the associated Banach $^*$-algebra crossed product $\ell^1 (\Sigma)$ and analyse its ideal structure. This algebra is the Banach algebra most naturally associated with the dynamical system, and it has a richer structure than its well studied $C^*$-envelope, as becomes evident from the possible existence of non-self-adjoint closed ideals. This paper initiates the study of these algebras and links their ideal structure to the topological dynamics. It is determined when exactly the algebra is simple, or prime, and when there exists a non-self-adjoint closed ideal. In addition, a structure theorem is obtained for the case when $X$ consists of one finite orbit, and the algebra is shown to be Hermitian if $X$ is finite. The key to these results lies in analysing the commutant of $C(X)$ in the algebra, which can be shown to be a maximal abelian subalgebra with non-zero intersection with each non-zero closed ideal.
\end{abstract}

\maketitle

\section{Introduction}\label{sec:introduction}

Whenever a locally compact group acts on a locally compact Hausdorff space, a $C^*$-algebra crossed product can be associated with this topological dynamical system, as a special case of the general $C^*$-crossed product construction for a group acting on a $C^*$-algebra \cite{williams}. If the space $X$ is compact, and the group consists of the integers acting via iterations of a given homeomorphism $\sigma$ of the space, the relation between the dynamics of the system $\Sigma=(X,\sigma)$ and the structure of the associated crossed product $C^*(\Sigma)$ is particularly well studied; see, e.g., \cite{tomiyama_book}, \cite{tomiyama_seoul_notes}, and \cite{tomiyama_kyoto_notes}, for a non-limitative introduction to the field and its authors.

The algebra $C^*(\Sigma)$, however, is not the Banach algebra most naturally associated with $\Sigma$. That predicate belongs to $\ell^1(\Sigma)$, the $\ell^1$-algebra of crossed product type of which $C^*(\Sigma)$ is the enveloping $C^*$-algebra. It is only natural to ask what the relation is between the structure of $\ell^1(\Sigma)$ and the dynamics of $\Sigma$, yet this matter has not been taken up so far. This is done in the present paper, which, apart from broadening our knowledge on the interplay between topological dynamics and Banach algebras, can also be seen as the study of a concrete class of involutive Banach algebras other than $C^*$-algebras, supplementing the general theory as can be found in, e.g., \cite{fragoulopoulou} and \cite{palmer}.  We show, for example, that $\ell^1(\Sigma)$ is simple precisely when $X$ is infinite and $\Sigma$ is minimal, and that it is prime precisely when $X$ is infinite and $\Sigma$ is topologically transitive. Quite contrary to $C^*(\Sigma)$, it is possible that $\ell^1(\Sigma)$ has a non-self-adjoint closed ideal: this is the case precisely when $\Sigma$ is not free.

At first sight, given the formal resemblance of some of the above results with those already known for $C^*(\Sigma)$, one might expect that the proofs are merely adaptations of the existing proofs for $C^*(\Sigma)$ to the situation of $\ell^1(\Sigma)$, which, after all, is also an involutive Banach algebra with a theory of states and Hilbert space representations available, strongly connected to that for $C^*(\Sigma)$. However, contrary to our own expectations, we have not been able to obtain satisfactory proofs along those lines, the reason for which lies in the fact that the notion of positivity in general Banach $^*$-algebras is more delicate than for $C^*$-algebras. Another conceivable attempt to benefit from what is already known for $C^*(\Sigma)$, would be to translate results on closed ideals of $C^*(\Sigma)$ back to results on closed ideals of $\ell^1(\Sigma)$. However, the relation between closed ideals in both algebras is presently not clear enough to make this work. For example: is it true that the closure in $C^*(\Sigma)$ of a proper closed ideal of $\ell^1(\Sigma)$ is always proper again? Clearly one would need to know this for such a translation approach, but at the time of writing this matter is open. The key proofs in the present paper, therefore, are fundamentally different from those for $C^*(\Sigma)$. States nor positivity are used, and, in fact, the involutive structure of $\ell^1(\Sigma)$ plays a very modest role in the proofs indeed.

The algebras $\ell^1(\Sigma)$ have a rich structure, much richer than their $C^*$-envelopes, a circumstance which may account for them having received relatively little attention thus far. For example, if $X$ consists of one point, then $\ell^1(\Sigma)$ is the usual group algebra $\ell^1(\mathbb Z)$, and its $C^*$-envelope is $C(\mathbb T)$. Whereas the latter can be considered as a rather accessible Banach algebra, the closed ideals of which are easily explicitly described, the algebra $\ell^1(\mathbb Z)$ is considerably more complicated, as is, e.g., demonstrated by the failure of spectral synthesis, and the existence of a non-self-adjoint closed ideal. It is therefore to be expected that, for general compact $X$, when the dynamics has an actual role to play, phenomena will emerge in $\ell^1(\Sigma)$, which do not occur in $C^*(\Sigma)$. An example of this can already be found in this paper, and is a generalisation of the result for $\ell^1(\mathbb Z)$ above: as already mentioned, $\ell^1(\Sigma)$ can have a non-self-adjoint closed ideal. Another intriguing non-$C^*$-question is whether $\ell^1(\Sigma)$ is always a Hermitian Banach $^*$-algebra, or, if not, which conditions on the dynamics are equivalent with this property. In this paper, we take a first step investigating this matter. Combining a structure theorem obtained in this paper with Wiener's classical result on the reciprocal of a function on the torus with an absolutely convergent Fourier series, we show that, for finite $X$, $\ell^1(\Sigma)$ is always Hermitian.

There was some a priori evidence available that it could be possible to obtain relations between the structure of algebras of crossed product type associated with $\Sigma$ and the dynamics of $\Sigma$, outside the $C^*$-context. Indeed, it is shown in \cite{SSdJ_IJM}, \cite{SSdJ_Banff}, and \cite{SSdJ_Lund}, that a number of connections between topological dynamics and $C^*$-algebras in the literature have an analogue for a certain $^*$-algebra $c_{00} (\Sigma)$, which is dense in $\ell^1(\Sigma)$ and (hence) in $C^* (\Sigma)$. Although these results are of a purely algebraic nature, they hint that results in this vein may be obtainable in a broader analytical context than $C^*(\Sigma)$. The present paper may serve to show that, with new techniques, this is actually possible.

We conclude this introduction with an overview of the paper.

In Section~\ref{sec:preliminaries} we collect the basic definitions and notations, introduce two algebras associated with $\Sigma$ and representations thereof, and include a few technical preparations on principal ideals.

Section~\ref{sec:commutant} is the mathematical backbone of the paper. The commutant of $C(X)$ in $\ell^1(\Sigma)$ is analysed, and it is shown that it is a maximal abelian subalgebra having non-zero intersection with every non-zero closed ideal of $\ell^1(\Sigma)$. The latter important property, which requires some effort to establish, was discovered by Svensson in an algebraic context, where the proof is much easier, cf.\ \cite[Theorem~6.1]{SSdJ_Lund}, and \cite[Theorem~3.1]{SSdJ_Banff}. It has an analogue for $C^*(\Sigma)$, cf.~\cite[Corollary~5.4]{ST_JFA}, and has provided a fruitful angle to study the structure of algebras associated with $\Sigma$.

Section~\ref{sec:ideal_structure} contains the actual results on the relation between the ideal structure of $\ell^1(\Sigma)$ and the dynamics of $\Sigma$. The intersection property for the commutant of $C(X)$ is easily translated to a condition (topologically freeness of $\Sigma$) for the similar property to hold for $C(X)$ itself, and this in turn is a vital ingredient for the results in the remainder of that Section. This logical build-up is inspired by  \cite{SSdJ_Banff}, \cite{SSdJ_Lund}, and \cite{ST_JFA}. Amongst others, the existence of a non-self-adjoint closed ideal of $\ell^1(\Sigma)$, and simplicity and primeness of $\ell^1(\Sigma)$ are considered. Since closed ideals are no longer necessarily self-adjoint, there are now also natural notions of $^*$-simplicity and $^*$-primeness, but, interestingly enough, these notions turn out to coincide with the non-involutive notions. We also obtain a structure theorem for $\ell^1(\Sigma)$ when $X$ consists of one finite orbit, which, when combined with Wiener's classical result as already mentioned, implies that $\ell^1(\Sigma)$ is Hermitian if $X$ is finite.
 
\section{Definitions and preliminaries}\label{sec:preliminaries}

In this section, we collect a number of definitions and preliminary results on the dynamics of a topological system, two involutive algebras associated with such a system, and representations of these algebras.
Included are also technical preparations on principal ideals of these algebras, which will be an important tool when establishing the main result in Section~\ref{sec:commutant}, Theorem~\ref{thm:commutant_intersection_property}. In view of the technical nature of these preparations, we have included them in this preliminary section in order not to interrupt the exposition later on.

Throughout this paper, $X$ denotes a non-empty compact Hausdorff space. For a subset $S$ of $X$ we write its interior as $S^\circ$, its closure as $\bar{S}$, and its complement as $S^c$. We will be concerned with a topological dynamical system $\Sigma = (X, \sigma)$, where $\sigma : X \rightarrow X$ is a homeomorphism.

\subsection{Dynamics}\label{subsec:dynamics}
We let $\mathbb{Z}$ act on $X$ via iterations of $\sigma$, and, for $n\in\mathbb Z$, we let $\Fix_n(\sigma) = \{ x \in X : \sigma^n(x) = x\} = \Fix_{-n}(\sigma)$ denote the fixed points of $\sigma^n$. Note that the sets $\Fix_n(\sigma)$ are closed for all $n$, and that they are invariant under the $\mathbb Z$-action. For $n\geq 1$, we let $\Per_n(\sigma)$ denote the set of points with period precisely $n$. The sets $\Per_n(\sigma)$ are invariant under the $\mathbb Z$-action. If $x\in\Fix_n(\sigma)$, for some $n\neq 0$, then $x\in\Fix_{jn}(\sigma)$ for all $j\in\mathbb Z$, and $x\in\Per_k(\sigma)$ for a unique $k\geq 1$; this $k$ divides $n$. Let $\Per(\sigma)=\bigcup_{n=1}^\infty \Fix_n(\sigma)$ be the set of periodic points, and let $\Aper(\sigma)$ be the set of aperiodic points. Hence $X=\Aper(\sigma)\cup\Per(\sigma)=\Aper(\sigma)\cup\bigcup_{n=1}^\infty \Fix_n(\sigma)=\Aper(\sigma)\bigcup_{n=1}^\infty \Per_n(\sigma)$.

If $\Aper(\sigma) = X$, then $\Sigma$ is called free, and if $\Aper(\sigma)$ is dense in $X$, then $\Sigma$ is called topologically free.

For a point $x \in X$, we denote by $\mathcal{O}_{\sigma} (x) = \{\sigma^n (x) : n \in \mathbb{Z}\}$ the orbit of $x$, and we recall that $\Sigma$ is called minimal if $\mathcal{O}_\sigma$ is dense in $X$, for every $x$ in $X$. The system is called topologically transitive if, for any pair of non-empty open subsets $U, V$ of $X$, there exists an integer $n$ such that $\sigma^n (U) \cap V \neq \emptyset$.

The following topological lemma will be used several times in this paper. It is based on the category theorem for (locally) compact Hausdorff spaces \cite[Theorem 2.2]{rudin_functional_analysis}. We include the easy proof of the first part (which occurs as Lemma~3.1 in \cite{ST_JFA}) for the convenience of the reader, and prefer to give a direct argument for the second part, rather than considering it as a consequence of the third part, the proof of which is more involved.

\begin{lemma}\label{lem:topological} \quad
\begin{enumerate}
\item $X$ is topologically free if and only if $\Fix_n(\sigma)^\circ=\emptyset$ for all $n\geq 1$.
\item The union of $\Aper(\sigma)$ and $\bigcup_{n=1}^\infty \Fix_n(\sigma)^\circ$ is dense in $X$.
\item The union of $\Aper(\sigma)$ and $\bigcup_{n=1}^\infty \Per_n(\sigma)^\circ$ is dense in $X$.
\end{enumerate}
\end{lemma}

\begin{proof} As to the first part, note that $\Aper(\sigma)=\bigcap_{n=1}^\infty \Fix_n(\sigma)^c$. Since the sets $\Fix_n(\sigma)^c$ are all open, the category theorem yields that $\Aper(\sigma)$ is dense if and only if $\Fix_n(\sigma)^c$ is dense for all $n\geq 1$, i.e., if and only if $\Fix_n(\sigma)^\circ=\emptyset$ for all $n\geq 1$.

As to the second part, let
\[
Y=\overline{\Aper(\sigma)\cup\bigcup_{n=1}^\infty \Fix_n(\sigma)^\circ}^{\,\,c}.
\]
Suppose that $Y\neq\emptyset$. Since $Y\subset\Per(\sigma)$, $Y=\bigcup_{n=1}^\infty Y\cap\Fix_{n}(\sigma)$. Now $Y$, being open in $X$, is a locally compact Hausdorff space in the induced topology, so the category theorem shows that there exists $n_0\geq 1$ such that the $Y$-closure of $Y\cap\Fix_{n_0}(\sigma)$ has non-empty $Y$-interior. Since $\Fix_{n_0}(\sigma)$ is closed in $X$, $Y\cap\Fix_{n_0}(\sigma)$ is closed in $Y$, and since $Y$ is open in $X$, a $Y$-open subset of $Y$ is open in $X$. We conclude that there exists a non-empty open subset $U$ of $X$ such that $U\subset Y\cap\Fix_{n_0}(\sigma)$. Hence $Y\cap\Fix_{n_0}(\sigma)^\circ\supset U\neq\emptyset$, which contradicts that $Y\cap\Fix_n(\sigma)^\circ=\emptyset$ for all $n\geq 1$ by construction.

The statement in the third part is \cite[Lemma~2.1]{ST_JFA}, and we refer to that paper for its proof.
\end{proof}

\subsection{Algebras}\label{subsec:algebras}
Let $\alpha$ be the automorphism of $C(X)$ induced by $\sigma$ via $\alpha(f) = f \circ \sigma^{-1}$, for $f \in C(X)$. Via $n \mapsto \alpha^n$, the integers act on $C(X)$ by iterations.
Given a topological dynamical system $\Sigma$, we endow the Banach space
\[
\ell^1 (\Sigma) = \{a: \mathbb{Z} \rightarrow C(X) : \Vert a\Vert= \sum_{k \in \mathbb{Z}} \|a(k)\|_{\infty} < \infty\},
\]
where $\|\cdot\|_{\infty}$ denotes the supremum norm on $C(X)$, with a multiplication and an involution. It will then be a unital Banach $^*$-algebra of crossed product type with an isometric involution.
The Banach space structure on $\ell^1 (\Sigma)$ is the natural pointwise one, and multiplication is defined by twisted convolution, as follows:
\[
(ab) (n) = \sum_{k \in \mathbb{Z}} a(k) \cdot \alpha^k (b(n-k)),
\]
for $a, b \in \ell^1 (\Sigma)$. We define the involution, $^*$, by
\[a^* (n) = \overline{\alpha^n (a(-n))},\]
for $a \in \ell^1 (\Sigma)$. The bar denotes the usual pointwise complex conjugation on $C(X)$. It is then routine to check that, when endowed with these operations, $\ell^1 (\Sigma)$ is indeed a unital Banach $^*$-algebra with isometric involution, and that the norm of the identity element (which maps $0\in\mathbb Z$ to $1\in C(X)$ and is zero elsewhere) is equal to one.

A useful way of working with $\ell^1 (\Sigma)$ is provided by the following. For $n,m \in \mathbb{Z}$, let
\begin{equation*}
  \chi_{\{n\}} (m) =
  \begin{cases}
    1 &\text{if }m =n;\\
    0 &\text{if }m \neq n,
  \end{cases}
\end{equation*}
where the constants denote the corresponding constant functions in $C(X)$. Then $\chi_{\{0\}}$ is the identity element of $\ell^1(\Sigma)$. Let $\delta = \chi_{\{1\}}$; then $\chi_{\{-1\}}=\delta^{-1}=\delta^*$. If we put $\delta^0=\chi_{\{0\}}$, one easily sees that $\delta^n = \chi_{\{n\}}$, for all $n \in \mathbb{Z}$. Hence $\Vert\delta^n\Vert=1$, for all $n\in\mathbb Z$. We may canonically view $C(X)$ as a closed abelian $^*$-subalgebra of $\ell^1 (\Sigma)$, namely as $\{a \delta^0 \, : \, a \in C(X)\}$. If $a \in \ell^1 (\Sigma)$, and if we write $a_k = a(k)$ for short, then $a= \sum_{k \in \mathbb{Z}} a_k\chi_{\{0\}} \chi_{\{k\}}$, where the series is absolutely convergent in $\ell^1(\Sigma)$. Hence, if we identify $a_k\chi_{\{0\}}\in\ell^1(\Sigma)$ and $a_k\in C(X)$, we have $a=\sum_{k\in\mathbb Z}a_k\delta^k$ as an absolutely convergent series in $\ell^1(\Sigma)$.
In the rest of this paper we will constantly use this expansion $a= \sum_{k\in\mathbb Z}a_k\delta^k$ of an arbitrary element $a\in\ell^1(\Sigma)$ as an absolutely convergent series, and we note that the $a_k$ are uniquely determined.
Thus $\ell^1 (\Sigma)$ is generated as a unital Banach algebra by an isometrically isomorphic copy of $C(X)$ and the elements $\delta$ and $\delta^{-1}$, subject to the relation $\delta f \delta^{-1} = \alpha(f)=f\circ\sigma^{-1}$, for $f \in C(X)$. The isometric involution is determined by $f^*=\overline f$, for $f\in C(X)$, and $\delta^*=\delta^{-1}$.

We let $c_{00}(\Sigma)$ denote the elements of $\ell^1(\Sigma)$ with finite support, i.e., the elements of the form $a=\sum_{k\in S} f_k \delta^k$, where $S\subset\mathbb Z$ is finite. This algebra and its generalisations are studied in \cite{SSdJ_IJM}, \cite{SSdJ_Banff}, and
\cite{SSdJ_Lund}. The algebra $c_{00}(\Sigma)$ is a dense unital $^*$-subalgebra of $\ell^1 (\Sigma)$, and as an associative algebra it is generated by an isometrically $^*$-isomorphic copy of $C(X)$, and the elements $\delta$ and $\delta^{-1}$, subject again to the relation $\delta f \delta^{-1} = \alpha(f)=f\circ\sigma^{-1}$, for $f \in C(X)$. The isometric involution is again determined by $f^*=\overline f$, for $f\in C(X)$, and $\delta^*=\delta^{-1}$.

We let $E:\ell^1(\Sigma)\to C(X)$ denote the canonical unital norm one projection, defined by $E(a)=a_0$, for $a=\sum_{k\in\mathbb Z}a_k\delta^k\in\ell^1(\Sigma)$. Note that $E(fag)=fE(a)g$, for all $a\in\ell^1(\Sigma)$, and $f,g\in C(X)$. If $a= \sum_{k \in \mathbb{Z}} a_k \delta^k\in\ell^1(\Sigma)$, then $E(a^*a)=\sum_{k\in\mathbb Z}|f_k\circ\sigma^k|^2$. Hence $E(a^*a)=0$ if and only if $a=0$. Note also that if $I$ is an ideal and $E(I) = \{0\}$, then $I = \{0\}$. Indeed, for arbitrary $a\in I$ and $n\in\mathbb Z$ one then has $a_n = E(a\delta^{-n}) = 0$.

In the sequel, an ideal is always assumed to be two-sided, but not necessarily self-adjoint or closed.

\subsection{Representations}\label{subsec:representations}

Two families of unital contractive $^*$-representations of $\ell^1(\Sigma$) are naturally associated with the dynamics of the system, and we will now describe these. They will be used in Section~\ref{sec:ideal_structure}.

Firstly, for each $x\in X$, an infinite dimensional $^*$-representation $\pi_x$ in a separable Hilbert space $H$ with orthonormal basis $\{e_j\}_{j\in\mathbb Z}$ and bounded operators $B(H)$ can be defined, as follows. For $j\in\mathbb Z$ and $f\in C(X)$, let $\pi_x(f)e_j=f(\sigma^j x)e_j$. For $j\in\mathbb Z$, put $\pi_x(\delta)x_j=x_{j+1}$. Then $\pi_x(f)$ and $\pi_x(\delta)$ are bounded operators on $H$. Note that $\pi_x(\delta)$ is unitary, and, for $f\in C(X)$, that $\pi_x(\overline f)=\pi(f)^*$, and $\Vert \pi_x(f)\Vert\leq\Vert f\Vert_\infty$. For $a=\sum_{k \in \mathbb{Z}} a_k \delta^k\in\ell^1(\Sigma)$, define $\pi_x(a)=\sum_{k\in\mathbb Z} \pi_x(a_k)\pi_x(\delta)^k$. Since $\Vert \pi_x(a_k)\pi_x(\delta)^k\Vert\leq \Vert a_k\Vert_\infty$, the series is absolutely convergent in the operator norm, and $\pi_x:\ell^1(\Sigma)\to B(H)$ is contractive. A short calculation shows that $\pi_x$ preserves products of elements of the form $a_k\delta^k$ ($a_k\in C(X),\,k\in\mathbb Z$), which span $c_{00}(\Sigma)$. Hence $\pi_x:c_{00}(\Sigma)\to B(H)$ is a homomorphism. Moreover, since, for $f\in C(X)$, $\pi_x(\overline f)=\pi_x(f)^*$, and $\pi_x(\delta^*)=\pi_x(\delta^{-1})=\pi_x(\delta)^{-1}=\pi_x(\delta)^*$, as well as $\pi_x((\delta^{-1})^*)=\pi_x(\delta)=(\pi_x(\delta)^{-1})^*=\pi_x(\delta^{-1})^*$, $\pi_x$ preserves the involution for a set which generates $c_{00}(\Sigma)$ as an associative algebra. Hence it is a $^*$-homomorphism, and thus $\pi_x:c_{00}(\Sigma)\to B(H)$ is a unital contractive $^*$-representation of $c_{00}(\Sigma)$ in $H$. By density of $c_{00}(\Sigma)$ in $\ell^1(\Sigma)$, we conclude that, for $x\in X$, $\pi_x: \ell^1(\Sigma)\to B(H)$ is a unital contractive $^*$-representation of $\ell^1(\Sigma)$ in $H$.

Secondly, if $x\in\Fix_n(\sigma)$, for some $n\geq 1$, and $0\neq z\in\mathbb C$, then a representation $\pi_{x,n,z}$ of $c_{00}(\Sigma)$ in a Hilbert space $H_n$ with orthonormal basis $\{e_j\}_{j=0}^{n-1}$ and bounded operators $B(H_n)$ can be defined, as follows. For $j=0,\ldots,n-1$ and $f\in C(X)$, let $\pi_{x,n,z}(f)e_j=f(\sigma^j x)e_j$. For $j=0,\ldots,n-2$, put $\pi_{x,n,z}(\delta)e_j=e_{j+1}$, and let $\pi_{x,n,z}(\delta)e_{n-1}=ze_0$. Then $\pi_{x,n,z}(\delta)$ is invertible, so that, for $a=\sum_{k \in \mathbb{Z}} a_k \delta^k\in c_{00}(\Sigma)$, the definition $\pi_{x,n,z}(a)=\sum_{k\in\mathbb Z} \pi_{x,n,z}(a_k)\pi_{x,n,z}(\delta)^k$ is meaningful. Again, a short calculation shows that $\pi_{x,n,z}$ preserves products for elements of the form $a_k\delta^k$ ($a_k\in C(X),\,k\in\mathbb Z$), which span $c_{00}(\Sigma)$. Hence $\pi_{x,n,z}:c_{00}(\Sigma)\to B(H_n)$ is a homomorphism. Not all representations $\pi_{x,n,z}$ are bounded on $c_{00}(\Sigma)$. In fact, since, for $j\in\mathbb Z$, $\pi_{x,n,z}(\delta^{jn})=\pi_{x,n,z}(\delta)^{jn}= z^j\id_{H_n}$, for boundedness of the representation it is evidently necessary that $z\in\mathbb T$. This condition is also sufficient, so that $\pi_{x,n,z}$ is a bounded representation of $c_{00}(\Sigma)$ precisely when $z\in\mathbb T$. Indeed, for such $z$ one has $\Vert\pi_{x,n,z}(\delta)\Vert=\Vert\pi_{x,n,z}(\delta)^{-1}\Vert=1$, and this easily implies that $\pi_{x,n,z}$ is a contractive representation of $c_{00}(\Sigma)$. If $z\in\mathbb T$, and $a=\sum_{k \in \mathbb{Z}} a_k \delta^k\in\ell^1(\Sigma)$, the definition $\pi_{x,n,z}(a)=\sum_{k\in\mathbb Z} \pi_{x,n,z}(a_k)\pi_{x,n,z}(\delta)^k$ as a series which converges absolutely in the operator norm, yields a contractive map $\pi_{x,n,z}:\ell^1(\Sigma)\to B(H_n)$. Furthermore, if $z\in\mathbb T $, then $\pi_{x,n,z}(\delta)$ is unitary, and the same argument as for the representations $\pi_x$ above shows again that $\pi_{x,n,z}:c_{00}(\Sigma)\to B(H_n)$ is a $^*$-homomorphism. Thus $\pi_{x,n,z}:c_{00}(\Sigma)\to B(H)$ is a unital contractive $^*$-representation of $c_{00}(\Sigma)$ in $H_n$. By density of $c_{00}(\Sigma)$ in $\ell^1(\Sigma)$, we conclude that, for $n\geq 1$, $x\in\Fix_n(\sigma)$, and $z\in\mathbb T$, $\pi_{x,n,z}: \ell^1(\Sigma)\to B(H_n)$ is a unital contractive $^*$-representation of $\ell^1(\Sigma)$ in $H_n$.

\begin{remark}
Not all representations of $\ell^1(\Sigma)$ in these two families are irreducible. For example, if $X$ consists of one point $x_0$, so that $\ell^1(\Sigma)=\ell^1(\mathbb Z)$ is commutative and all irreducible representations must be one-dimensional, the only irreducible representations are the $\pi_{x_0,1,z}$, for $z\in\mathbb T$, which then correspond to evaluating the Fourier transform in $z$. In a future paper we will report separately on the irreducibility of these representations, and on their relation with the character space of the commutant of $C(X)$ in $\ell^1(\Sigma)$ and in $C^*(\Sigma)$. For the purpose of this paper, however, it is sufficient to merely establish the existence of the two families above.
\end{remark}

\subsection{Preparatory results on principal ideals}\label{subsec:technical_preparations}

As a preparation for the proofs leading to Theorem~\ref{thm:commutant_intersection_property}, we establish Corollary~\ref{cor:unimodular_functions} below. It is concerned with constructing a non-zero element in a principal ideal, such that certain of its coefficients are zero, its coefficient in degree zero is non-zero, and the supports of all coefficients are under control.

We start with a lemma.

\begin{lemma}\label{lem:unimodular_functions}
Let $x_0\in X$, $k_0\in\mathbb Z$, and $n_0\geq 0$. Suppose that $x_0\in\Fix_{n_0}(\sigma)^\circ$ and $\sigma^{k_0}(x_0)\neq x_0$. Then there exists an open neighbourhood $U$ of $x_0$, contained in $\Fix_{n_0}(\sigma)$, with the property that, for each $0\neq\lambda\in\mathbb C$, there exists a function $g\in C(X)$, such that $g$ is equal to 1 on $U$ and $g(x)\overline g (\sigma^{-k_0-jn_0}x)=\lambda$, for all $x\in U$, and all $j\in\mathbb Z$. If $\lambda\in\mathbb T$, then $g$ can be chosen to be unimodular on $X$.
\end{lemma}

\begin{proof}
Since $\sigma^{-k_0}x_0\neq x_0$, separating these points by two disjoint open sets and then applying \cite[Theorem~5.18]{kelley} to each of these yields disjoint closed subsets $C$ and $C^\prime$, and open subsets $U$ and $U^\prime$, such that $x_0\in U\subset C$ and $\sigma^{-k_0}x_0\in U^\prime\subset C^\prime$. Replacing $U$ with $U\cap\Fix_{n_0}(\sigma)^\circ$, we may additionally assume that $U\subset\Fix_{n_0}(\sigma)$. Moreover, replacing $U$ with $U\cap\sigma^{k_0}U^\prime$, we may assume that $\sigma^{-k_0}U\subset C^\prime$. Then $U$ has the required properties. Indeed, if $0\neq\lambda\in\mathbb C$, choose $\mu\in\mathbb C$ such that $\exp (-i\overline\mu)=\lambda$. Urysohn's lemma provides a continuous function $\tilde g:X\to\mathbb R$, such that $\tilde g$ is equal to 0 on $C$ and equal to 1 on $C^\prime$. Then $g=\exp (i\mu\tilde g)$ is equal to 1 on $U$, and $g$ is unimodular on $X$ if $\lambda\in\mathbb T$. Furthermore, since $\sigma^{-jn_0}x=x$, for all $x\in U$ and all $j\in\mathbb Z$, we find that, for such $x$ and $j$, $g(x)\overline g (\sigma^{-k_0-jn_0}x)=1\cdot\exp (-i\overline\mu \tilde g(\sigma^{-k_0}x))=\exp (-i\overline\mu)=\lambda$.
\end{proof}

Next, we show how to construct an element in a principal ideal, such that certain of its coefficients vanish in a given point, while preserving the coefficient in degree zero.

\begin{proposition}\label{prop:unimodular_functions}
Let $x_0\in X$ and $n_0\geq 0$ be such that $x_0\in\Fix_{n_0}(\sigma)^\circ$. Suppose that, for some $N\geq 1$ and $k_1,\ldots,k_N\in\mathbb Z$, the points $\sigma^{k_1}x_0,\ldots,\sigma^{k_N}x_0$ are all different from $x_0$. Then there exist an open neighbourhood $U$ of $x_0$, contained in $\Fix_{n_0}(\sigma)^\circ$, and unimodular functions $\theta_1,\ldots,\theta_{2^N}\in C(X)$, which are equal to 1 on $U$, with the following property: If $a= \sum_{k \in \mathbb{Z}} a_k \delta^k\in\ell^1(\Sigma)$ is arbitrary, and $\frac{1}{2^N}\sum_{l=1}^{2^N} \theta_l a\overline \theta_l=\sum_{k\in\mathbb Z}a^\prime_k\delta^k$, then
\begin{enumerate}
\item $a^\prime_0=a_0$;
\item $a^\prime_{k_l+jn_0}(x)=0$, for all $x\in U$, $l=1,\ldots,N$, and $j\in\mathbb Z$.
\end{enumerate}

\end{proposition}

\begin{proof}
As a preparation, one computes easily that
\begin{equation}\label{eq:basic_relation}
\frac{1}{2}\left(ga\overline g + \overline g a g\right) = \sum_{k\in\mathbb Z} a_k \cdot \textup{Re}\, (g\cdot (\overline g\circ\sigma^{-k}))\delta^k,
\end{equation}
for arbitrary $a= \sum_{k \in \mathbb{Z}} a_k \delta^k\in\ell^1(\Sigma)$, and $g\in C(X)$. We will use this relation to prove the statement by induction with respect to $N$.

For $N=1$, an application of Lemma~\ref{lem:unimodular_functions} (with $\lambda=i$) yields an open neighbourhood $U$ of $x_0$, contained in $\Fix_{n_0}(\sigma)^\circ$, and a unimodular function $g\in C(X)$, which is equal to 1 on $U$, such that $g(x)\overline g (\sigma^{-k_1-jn_0}x)=i$, for all $x\in U$, and $j\in\mathbb Z$. Hence \eqref{eq:basic_relation} shows that the statement holds for $N=1$, with $\theta_1=g$ and $\theta_2=\overline g$, where we note that the coefficient of $\delta^0$ in the right hand side of \eqref{eq:basic_relation} is $a_0$ unchanged, as a consequence of the unimodularity of $g$.

Assume then, that for $N\geq 2$ the statement holds for all $\sigma^{k_1}x_0,\ldots,\sigma^{k_{N-1}}x_0$ different from $x_0$, and that $\sigma^{k_1}x_0,\ldots,\sigma^{k_N}x_0$ are given, all different from $x_0$. By the induction hypothesis, there exist an open neighbourhood $\tilde U$ of $x_0$, contained in $\Fix_{n_0}(\sigma)^\circ$, and unimodular functions $\tilde\theta_1,\ldots,\tilde\theta_{2^{N-1}}\in C(X)$, which are equal to 1 on $\tilde U$, such that, if $a= \sum_{k \in \mathbb{Z}} a_k \delta^k\in\ell^1(\Sigma)$ is arbitrary, and $\frac{1}{2^{N-1}}\sum_{l=1}^{2^{N-1}} \tilde\theta_l a\overline {\tilde\theta}_l=\sum_{k\in\mathbb Z}\tilde a_k\delta^k$, then
\begin{enumerate}
\item $\tilde a_0=a_0$;
\item $\tilde a_{k_l+jn_0}(x)=0$, for all $x\in \tilde U$, $l=1,\ldots,N-1$, and $j\in\mathbb Z$.
\end{enumerate}
An application of Lemma~\ref{lem:unimodular_functions} (with $\lambda=i$) yields an open neighbourhood $V$ of $x_0$, contained in $\Fix_{n_0}(\sigma)^\circ$, and a unimodular function $g\in C(X)$, which is equal to 1 on $V$, such that $g(x)\overline g (\sigma^{-k_N-jn_0}x)=i$, for all $x\in V$, and $j\in\mathbb Z$. Using \eqref{eq:basic_relation} we find that
\begin{align}\label{eq:repeated_relation}
\frac{1}{2}&\left(g\left[\frac{1}{2^{N-1}}\sum_{l=1}^{2^{N-1}} \tilde\theta_l a\overline {\tilde\theta}_l\right]\overline g + \overline g\left[\frac{1}{2^{N-1}}\sum_{l=1}^{2^{N-1}} \tilde\theta_l a\overline {\tilde\theta}_l\right] g\right)\\
&=\frac{1}{2}\left(g\left[\sum_{k\in\mathbb Z}\tilde a_k\delta^k\right] \overline g + \overline g\left[\sum_{k\in\mathbb Z}\tilde a_k\delta^k\right] g \right)
=\sum_{k\in\mathbb Z} \tilde a_k \cdot \textup{Re}\, (g\cdot (\overline g\circ\sigma^{-k}))\delta^k.\notag
\end{align}
If we write the rightmost expression in $\eqref{eq:repeated_relation}$ as $\sum_{k \in \mathbb{Z}} a_k^\prime \delta^k$, then
\begin{enumerate}
\item $a_0^\prime=\tilde a_0=a_0$;
\item $a_{k_l+jn_0}^\prime(x)=0$, for all $x\in \tilde U$, $l=1,\ldots,N-1$, and $j\in\mathbb Z$, since $\tilde a_{k_l+jn_0}(x)=0$ for such $x,l$ and $j$;
\item $a_{k_N+jn_0}^\prime(x)=0$, for all $x\in V$, and $j\in\mathbb Z$, by the choice of $g$.
\end{enumerate}
Define $\theta_l=g\tilde\theta_l$ for $l=1,\ldots, 2^{N-1}$, and $\theta_l=\overline g\tilde\theta_{l-2^{N-1}}$ for $l=2^{N-1}+1,\ldots,2^N$. Then the left hand side in \eqref{eq:repeated_relation} is equal to $\frac{1}{2^N}\sum_{l=1}^{2^N} \theta_l a\overline\theta_l$, and we see that $U=\tilde U\cap V$ and $\theta_1,\ldots,\theta_{2^N}$ are as required.
\end{proof}

From the previous result, we can now infer the corollary we need.

\begin{corollary}\label{cor:unimodular_functions}
Let $x_0\in X$, $a=\sum_{k \in \mathbb{Z}} a_k \delta^k\in\ell^1(\Sigma)$, and suppose $a_0(x_0)\neq 0$.
\begin{enumerate}
\item Suppose that, for some $N\geq 1$ and $k_1,\ldots,k_N\in\mathbb Z$, the points $\sigma^{k_1}x_0,\ldots,\sigma^{k_N}x_0$ are all different from $x_0$. Then there exist an open neighbourhood $U$ of $x_0$, unimodular functions $\theta_1,\ldots,\theta_{2^N}\in C(X)$, which are equal to 1 on $U$, and a function $f\in C(X)$ with $0\leq f\leq 1$, $f(x_0)=1$, and support contained in $U$, such that, if $\frac{1}{2^N}\sum_{l=1}^{2^N} f\theta_l a \overline \theta_l=\sum_{k\in\mathbb Z}a^\prime_k\delta^k$, then
\begin{enumerate}
\item $a^\prime_0=fa_0\neq 0$;
\item $a^\prime_{k_l}=0$, for $l=1,\ldots,N$;
\item $a^\prime_k$ is supported in $U$, for all $k\in\mathbb Z$.
\end{enumerate}
\item If $x_0\in\Per_{n_0}(\sigma)^\circ$, for some $n_0\geq 2$, then there exist an open neighbourhood $U$ of $x_0$, contained in $\Per_{n_0}(\sigma)^\circ$, and unimodular functions $\theta_1,\ldots,\theta_{2^{n_0-1}}\in C(X)$, which are equal to 1 on $U$, and a function $f\in C(X)$ with $0\leq f\leq 1$, $f(x_0)=1$, and support contained in $U$, such that, if $\frac{1}{2^{n_0-1}}\sum_{l=1}^{2^{n_0-1}} f\theta_l a \overline \theta_l=\sum_{k\in\mathbb Z}a^\prime_k\delta^k$, then
\begin{enumerate}
\item $a^\prime_0=fa_0\neq 0$;
\item $a^\prime_{l+jn_0}=0$, for $l=1,\ldots,n_0 - 1$, and all $j\in\mathbb Z$.
\item $a^\prime_{jn_0}$ is supported in $U\subset\Per_{n_0}(\sigma)^\circ\subset\Fix_{n_0}(\sigma)\subset\Fix_{jn_0}(\sigma)$, for all $j\in\mathbb Z$.
\end{enumerate}
\end{enumerate}
\end{corollary}

\begin{proof}
For part (1), we start with an application of Proposition~\ref{prop:unimodular_functions} with $n_0=0$.  This yields an open neighbourhood $U$ of $x_0$ and unimodular functions $\theta_1,\ldots,\theta_{2^N}\in C(X)$, which are equal to one on $U$, such that, if $\frac{1}{2^N}\sum_{l=1}^{2^N} \theta_l a \overline \theta_l=\sum_{k\in\mathbb Z}\tilde a_k\delta^k$, then
\begin{enumerate}
\item[(a)] $\tilde a_0=a_0$;
\item[(b)] $\tilde a_{k_l}(x)=0$, for all $x\in U$, and $l=1,\ldots,N$.
\end{enumerate}
Next, choose $f\in C(X)$ with $0\leq f\leq 1$, $f(x_0)=1$, and supported in $U$. Then $\frac{1}{2^N}\sum_{l=1}^{2^N} f\theta_l a \overline \theta_l$ evidently has all the required properties, as $f \tilde a_{k_l}=0$, for $l=1,\ldots,N$. The statement for $c_{00}(\Sigma)$ is clear.

The proof of part (2) is a similar application of Proposition~\ref{prop:unimodular_functions}, which is applicable since $\Per_{n_0}(\sigma)^\circ\subset\Fix_{n_0}(\sigma)^\circ$ and the points $\sigma^1 x_0,\ldots, \sigma^{n_0-1}x_0$ are all different from $x_0$. One replaces $U$ as provided by Proposition~\ref{prop:unimodular_functions} with $U\cap\Per_{n_0}(\sigma)^\circ$, and subsequently chooses $f\in C(X)$ as above.
\end{proof}

\section{The commutant of $C(X)$}\label{sec:commutant}

The analysis of the commutant of $C(X)$ in $\ell^1 (\Sigma)$, denoted by $C(X)^\prime$,
and defined as
\[
C(X)^\prime = \left\{a \in \ell^1 (\Sigma) : af = fa \textup{ for all } f \in C(X)\right\},
\]
in this section is the basis for the results in Section~\ref{sec:ideal_structure}. Obviously, $C(X)^\prime$ is a unital Banach $^*$-subalgebra of $\ell^1 (\Sigma)$. What is less obvious, but not difficult to prove, is that it is actually commutative, and hence a maximal abelian subalgebra of $\ell^1(\Sigma)$ (Proposition~\ref{prop:commutant_abelian}). The actual landmark of this section, however, is the result that $C(X)^\prime\cap I\neq\{0\}$, for every non-zero closed ideal $I$ of $\ell^1(\Sigma)$ (Theorem~\ref{thm:commutant_intersection_property}).

We will now set out to establish these results, and we start with the following concrete description of $C(X)^\prime$.

\begin{proposition}\label{prop:commutant_description}
$C(X)^\prime = \{\sum_{k\in\mathbb Z} a_k \delta^k \in \ell^1 (\Sigma) : \supp (a_k)\subset \Fix_k(\sigma)  \textup{ for all }k\in\mathbb Z\}$.
Consequently, $C(X)^\prime= C(X)$ if and only if the dynamical system is topologically free.
\end{proposition}

\begin{proof}
The assertion is an adaptation of~\cite[Corollary 3.4]{SSdJ_IJM} to our context, and we include a proof here for the reader's convenience. Suppose $a = \sum_{k\in\mathbb Z} a_k \delta^k \in C(X)^\prime$. For any $f$ in $C(X)$ we have $fa=\sum_{k\in\mathbb Z} fa_k \delta^k$ and $af=\sum_{k\in\mathbb Z} a_k \delta^k f=\sum_{k\in\mathbb Z} a_k(f\circ\sigma^{-k}) \delta^k$. Hence $a\in C(X)^\prime$ if and only if $f(x)a_k(x)=f(\sigma^{-k}x)a_k(x)$, for all $k\in\mathbb Z$, $f\in C(X)$, and $x\in X$. Therefore, if $a_k (x)$ is non zero we have $f(x) =
 f(\sigma^{-k} x)$, for all $f \in C(X)$. It follows that $\sigma^{-k} x = x$, i.e.,\ $x$ belongs to $\Fix_k(\sigma)$. Since $\Fix_k (\sigma)$ is closed, $\supp(a_k)\subset \Fix_k(\sigma)$. Conversely, if $\supp(a_k) \subset \Fix_k (\sigma)$ for all $k\in\mathbb Z$, then $f(x)a_k(x)=f(\sigma^{-k}x)a_k(x)$, for all $k\in\mathbb Z$, $f\in C(X)$, and $x\in X$. This establishes the description of $C(X)^\prime$.

As to the remaining part of the statement, by Lemma~\ref{lem:topological}, $\Sigma$ is topologically free if and only if for every non-zero integer $k$ the set $\Fix_k(\sigma)$ has empty interior. So, when the system is topologically free, we see from the above description of $C(X)^\prime$ that an element $a$ of $C(X)^\prime$ necessarily belongs to $C(X)$. If $\Sigma$ is not topologically free, however, $\Fix_k (\sigma)$ has non-empty interior, for some non-zero $k$, and hence there is a non-zero function $f \in C(X)$, such that $\supp(f) \subset \Fix_k (\sigma)$. Then $f \delta^k \in C(X)^\prime$ by the above, but $f\delta^k\notin C(X)$.
\end{proof}

The following result is an adaptation of \cite[Proposition 2.1]{SSdJ_IJM} to our set-up. In spite of its elementary proof, it may come as a surprise that $C(X)^\prime$ is a maximal abelian subalgebra.

\begin{proposition}\label{prop:commutant_abelian}
The commutant $C(X)^\prime$ of $C(X)$ is abelian. In fact, it is the largest abelian subalgebra of $\ell^1 (\Sigma)$ containing $C(X)$, and it is a commutative unital Banach $^*$-subalgebra of $\ell^1(\Sigma)$.
\end{proposition}

\begin{proof}
We need only prove the first statement, since the rest is then clear. Suppose $a, b \in C(X)^\prime$. By definition of the multiplication in $\ell^1 (\Sigma)$ we have, for $n\in\mathbb Z$, $(ab)_n = \sum_{k\in\mathbb Z} a_k \cdot \alpha^k(b_{n-k})$. As $a \in C(X)^\prime$, it follows from Proposition~\ref{prop:commutant_description} that $a_k\cdot \alpha^{k}(b_{n-k})=a_k \cdot b_{n-k}$, for all $k\in\mathbb Z$. Hence $(ab)_n = \sum_{k\in\mathbb Z} a_k \cdot b_{n-k}$. Similarly, $(ba)_n = \sum_{k\in\mathbb Z} b_k \cdot a_{n-k}$. Thus $(ab)_n = (ba)_n$ for all $n\in\mathbb Z$, and hence $ab = ba$.
\end{proof}

We will now proceed towards the main result of this section, Theorem~\ref{thm:commutant_intersection_property}, benefiting from our main technical preparatory result, Corollary~\ref{cor:unimodular_functions}. We start with a relatively easy case of an algebraic nature on principal ideals generated by certain elements. We emphasize that the ideals under consideration in the following result and its corollary are not assumed to be closed.

\begin{proposition}
Let $a=\sum_{k \in \mathbb{Z}} a_k \delta^k\in\ell^1(\Sigma)$, and suppose that $a_{k_0}(x_0)\neq 0$, for some $x_0\in X$ and $k_0\in\mathbb Z$, and that $x_0\subset\Per_{n_0}(\sigma)^\circ$, for some $n_0\geq 1$. Then the principal ideal generated by $a$ in $\ell^1(\Sigma)$ has non-zero intersection with $C(X)^\prime$. In fact, there exist an open neighbourhood $U$ of $x_0$, contained in $\Per_{n_0}(\sigma)^\circ$, and unimodular functions $\theta_1,\ldots,\theta_{2^{n_0-1}}\in C(X)$, which are equal to 1 on $U$, and a function $f\in C(X)$ with $0\leq f\leq 1$, $f(x_0)=1$, and support contained in $U$, such that $\frac{1}{2^{n_0-1}}\sum_{l=1}^{2^{n_0-1}} f\theta_l a\delta^{-k_0} \overline \theta_l\in C(X)^\prime$, and $E(\frac{1}{2^{n_0-1}}\sum_{l=1}^{2^{n_0-1}} f\theta_l a\delta^{-k_0} \overline \theta_l)\neq 0$.
\end{proposition}

\begin{proof}
If $n_0\geq 2$, then one applies the second part of Corollary~\ref{cor:unimodular_functions} to $a\delta^{-k_0}$ to construct an element of the given form in the principal ideal generated by $a$ which has non-trivial coefficient in dimension zero and which, by Proposition~\ref{prop:commutant_description}, is in fact in $C(X)^\prime$. If $n_0=1$, one takes $U=\Per_1(\sigma)^\circ$, and chooses $f\in C(X)$, with $0\leq f\leq 1$, $f(x_0)=1$, and support contained in $U=\Per_1(\sigma)^\circ$. With $\theta_1=1$ we have $\frac{1}{2^{1-1}}\sum_{l=1}^{2^{1-1}} f\theta_l a \delta^{-k_0}\overline\theta_l=fa\delta^{-k_0}$. Since the support of each coefficient of $fa\delta^{-k_0}$ is contained in the support of $f$, which is contained in $\Per_{1}(\sigma)^\circ$, hence in $\Fix_n(\sigma)$, for all $n$, Proposition~\ref{prop:commutant_description} yields that $fa\delta^{-k_0}$ is in $C(X)^\prime$. Clearly its coefficient in dimension zero does not vanish.
\end{proof}

\begin{corollary}\label{cor:intersection_with_algebraic_ideals}
Let $I$ be an ideal of $\ell^1(\Sigma)$ such that $I\cap C(X)^\prime=0$. If $\sum_{k \in \mathbb{Z}} a_k \delta^k\in I$, then $a_k(x)=0$, for all $k\in\mathbb Z$, and all $x\in\overline{\bigcup_{n=1}^\infty \Per_n(\sigma)^\circ}$.
\end{corollary}

Hence, if $\overline{\bigcup_{n=1}^\infty \Per_n(\sigma)^\circ}=X$, as occurs, e.g., for rational rotations, then $C(X)^\prime$ has non-zero intersection with all non-zero ideals of $\ell^1(\Sigma)$, closed or not. On the other hand, if, e.g., the system is topologically free, Corollary~\ref{cor:intersection_with_algebraic_ideals} gives no information at all.

We will now establish a counterpart of Corollary~\ref{cor:intersection_with_algebraic_ideals} for closed ideals and aperiodic points. In its proof we will need the following result on the minimality of the $C^*$-norm for commutative $C^*$-algebras.

\begin{theorem}\label{thm:minimal_norm}
Let $X$ be a locally compact Hausdorff space. If $\Vert\,.\,\Vert$ is any norm under which $C_0(X)$ is a normed algebra, then $\Vert f\Vert_\infty\leq\Vert f\Vert$, for all $f\in C_0(X)$.
\end{theorem}

For the proof we refer to \cite[Theorem~1.2.4]{sakai}. Alternatively, one may use the more general result that on a semisimple regular commutative Banach algebra the spectral radius is dominated by every algebra norm \cite[Corollary 4.2.18]{kaniuth}. We note explicitly that the statement holds without any further assumption on completeness, the relation between $\Vert\,.\,\Vert$ and the canonical involution on $C(X)$, or the norm of the unit element of $C_0(X)$ in case $X$ is compact: the submultiplicativity of $\Vert\,.\,\Vert$ is all that is needed.

\begin{proposition}\label{prop:intersection_with_closed_ideals}
Let $I$ be a closed ideal of $\ell^1(\Sigma)$ such that $I\cap C(X)^\prime=0$. If $\sum_{k \in \mathbb{Z}} a_k \delta^k\in I$, then $a_k(x)=0$, for all $k\in\mathbb Z$, and all $x\in\Aper(\sigma)$.
\end{proposition}

\begin{proof}
As a preparation, let $q:\ell^1(\Sigma)\to\ell^1(\Sigma)/I$ be the quotient map. Since $I$ is closed, $\ell^1(\Sigma)/I$ is a normed algebra. We note that $I\cap C(X)\subset I\cap C(X)^\prime=\emptyset$; hence $q$ yields an embedding of $C(X)$ into $\ell^1(\Sigma)/I$. From Theorem~\ref{thm:minimal_norm} we conclude that $||q(f)||\geq ||f||_\infty$, for all $f\in C(X)$. Therefore $q$ is an isometric embedding of $C(X)$ into $\ell^1(\Sigma)/I$.

Suppose, then, that $a=\sum_{k \in \mathbb{Z}} a_k \delta^k\in I$, and that $x_0\in\Aper(\sigma)$. We will show that $a_{k_0}(x_0)=0$, for all $k_0\in\mathbb Z$. Replacing $a$ with $a\delta^{-k_0}$, we see that it is sufficient to show that $a_0(x_0)=0$. Let $\epsilon>0$. Choose $b,c\in\ell^1(\Sigma)$, and $n\geq 1$, such that $a=a_0+b+c$, with $b=\sum_{k=-n}^{-1} a_k\delta^k + \sum_{k=1}^ n a_k\delta^k$, and $||c||<\epsilon$. Since $x_0\in\Aper(\sigma)$, the points $\sigma^k x_0$, for $k=-n,\ldots,-1,1,\ldots,n$, are all different from $x_0$. Therefore, the first part of Corollary~\ref{cor:unimodular_functions} provides finitely many unimodular functions $\theta_1,\ldots,\theta_M$, and a function $f\in C(X)$, with $0\leq f\leq 1$, and $f(x_0)=1$, such that
\begin{align*}
\tilde a&=\frac{1}{M}\sum_{l=1}^M f \theta_l a \overline\theta_l\\
&=\frac{1}{M}\sum_{l=1}^M f \theta_l a_0 \overline\theta_l + \frac{1}{M}\sum_{l=1}^M f \theta_l b \overline\theta_l + \frac{1}{M}\sum_{l=1}^M f \theta_l c \overline\theta_l\\
&=fa_0 + \frac{1}{M}\sum_{l=1}^M f \theta_l c\overline\theta_l,
\end{align*}
where the term corresponding to $b$ has vanished on account of property (b) in the first part of Corollary~\ref{cor:unimodular_functions}. Write $\tilde c=\frac{1}{M}\sum_{l=1}^M f \theta_l c\overline\theta_l$ for short, so that $\tilde a=fa_0+\tilde c$. We note that $||\tilde c||\leq ||f||_\infty ||c||=||c||<\epsilon$. Since $\tilde a\in I$, $q(fa_0)=-q(\tilde c)$. Therefore, using the isometric character of $q$ on $C(X)$ in the third step, we find that
\[
|a_0(x_0)|=|f(x_0)a_0(x_0)|\leq ||fa_0||_\infty=||q(fa_0)||=||q(\tilde c)||\leq ||\tilde c||< \epsilon.
\]
Since $\epsilon>0$ was arbitrary, the proof is complete.
\end{proof}

Finally, our efforts are rewarded. The following result, on which the remainder of the paper rests, is based on all material presented so far, with the exception of the representations in Section~\ref{subsec:representations}.

\begin{theorem}\label{thm:commutant_intersection_property}
$C(X)^\prime\cap I\neq\{0\}$, for every non-zero closed ideal $I$ of $\ell^1(\Sigma)$.
\end{theorem}

\begin{proof}
Suppose $I$ is a closed ideal of $\ell^1(\Sigma)$, such that $I\cap C(X)^\prime=0$. If $a=\sum_{k \in \mathbb{Z}} a_k \delta^k\in I$, then Corollary~\ref{cor:intersection_with_algebraic_ideals} shows that the $a_k$ all vanish on $\bigcup_{n=1}^\infty \Per_n(\sigma)^\circ$. Since $I$ is additionally assumed to be closed, Proposition~\ref{prop:intersection_with_closed_ideals} also applies, showing that the $a_k$ all vanish on $\Aper(\sigma)$. Since $\Aper(\sigma)\cup \bigcup_{n=1}^\infty \Per_n(\sigma)^\circ$ is dense by the third part of Lemma~\ref{lem:topological}, all $a_k$ are identically zero. Hence $I$ is the zero ideal, as was to be proved.
\end{proof}

\begin{remark} It is also true that the commutant of $C(X)$ in $c_{00}(\Sigma)$ has non-zero intersection with each non-zero ideal. A result implying this was first obtained as \cite[Theorem~6.1]{SSdJ_Lund}, and later an even shorter proof of a more general statement was found, cf.\ \cite[Theorem~3.1]{SSdJ_Banff}. Needless to say, our proof of the analogous result in a truly analytical setting is considerably more involved.

In $C^*(\Sigma)$, the commutant of $C(X)$ has non-zero intersection with each non-zero ideal, not necessarily closed or self-adjoint \cite[Corollary~4.4]{ST_JFA}. The proof of that result relies extensively on the theory of states of a $C^*$-algebra and of its Pedersen ideal. It is therefore quite different from the above approach.
\end{remark}

\section{The ideal structure of $\ell^1(\Sigma)$}\label{sec:ideal_structure}

The key result Theorem~\ref{thm:commutant_intersection_property} allows us to prove a number of theorems relating the ideal structure of $\ell^1(\Sigma)$ and the topological dynamics of $\Sigma$. The analogues of these are known to hold in the interplay between $\Sigma$ and $C^* (\Sigma)$, except, naturally, the result on the possible existence of non-self-adjoint closed ideals. With Theorem~\ref{thm:commutant_intersection_property} at our disposal, and the representations in Section~\ref{subsec:representations} available, the proofs run rather smoothly. We are consecutively concerned with:
\begin{itemize}
\item Determining when $C(X)$ has non-zero intersection with each non-zero (self-adjoint) closed ideal (Theorem~\ref{thm:main_four_equivalences}), which in turn is an important ingredient for the sequel);
\item Determining when $\ell^1(\Sigma)$ is ($^*$-)simple (Theorem~\ref{thm:simple_minimality});
\item Determining when $\ell^1(\Sigma)$ has only self-adjoint closed ideals (Theorem~\ref{thm:free_self_adjoint});
\item A structure theorem for $\ell^1(\Sigma)$ when $X$ consists of one finite orbit (Theorem~\ref{thm:structure_theorem_for_one_orbit};
\item Showing that $\ell^1(\Sigma)$ is Hermitian if $X$ consists of a finite number of points (Theorem~\ref{thm:finite_hermitian});
\item Determining when $\ell^1(\Sigma)$ is ($^*$-)prime (Theorem~\ref{thm:prime_transitive}).
\end{itemize}

\begin{theorem}\label{thm:main_four_equivalences}
The following are equivalent:
\begin{enumerate}

\item $I\cap C(X) \neq 0$ for every non-zero closed ideal $I$ of $\ell^1 (\Sigma)$;
\item $I\cap C(X) \neq 0$ for every non-zero self-adjoint closed ideal $I$ of $\ell^1 (\Sigma)$;
\item $C(X)$ is a maximal abelian subalgebra of $\ell^1 (\Sigma)$;
\item $\Sigma$ is topologically free.
\end{enumerate}
\end{theorem}

\begin{proof}
Equivalence of (3) and (4) is an immediate consequence of Proposition~\ref{prop:commutant_description}, together with Proposition~\ref{prop:commutant_abelian}. As to the remaining implications, (1) evidently implies (2).

To show that (2) implies (4), we use the same technique as in the proof of \cite[Theorem 5.4]{tomiyama_seoul_notes}, based on representations.
Suppose that $\Sigma$ is not topologically free. Then, by the first part of Lemma~\ref{lem:topological}, there exists $n_0\geq 1$ such that $\Fix_{n_0} (\sigma)$ has non-empty interior. Let $f \in C(X)$ be non-zero and such that $\supp(f) \subset \Fix_{n_0}(\sigma)$, and consider the non-zero self-adjoint closed ideal $I$ of $\ell^1 (\Sigma)$, generated by $f - f \delta^{n_0}$. We claim that, for all $x\in\Fix_{n_0}(\sigma)^c$, the representation $\pi_x$ from Section~\ref{subsec:representations} vanishes on $I$, and that the same holds for the representation $\pi_{x,n_0,1}$, for all $x\in\Fix_n(\sigma)$. Assuming this for the moment, suppose that $g\in I\cap C(X)$. Then these representations all vanish on $g$. Since in all these representations the action of $g$ on the basis vector $e_0$ in the pertinent Hilbert space is multiplication with $g(x)$, we conclude that $g$ vanishes on $\Fix_{n_0}(\sigma)^c\cup\Fix_{n_0}(\sigma)=X$. Thus $I\cap C(X)=0$, and hence (2) implies (4).
Thus it remains to prove our claim. For this, it is sufficient to prove that these continuous $^*$-representations all vanish on $f-f\delta^{n_0}$. To start with, we note that, for all $x\in\Fix_{n_0}(\sigma)^c$, $\pi_x(f)=0$. Indeed, $\pi_x(f)$ is a diagonal operator with entries $f(\sigma^j x)$, for $j\in\mathbb Z$, but, if $x\in\Fix_{n_0}(\sigma)^c$, then $\sigma^j x\in\Fix_{n_0}(\sigma)^c$, for all $j\in\mathbb Z$, whereas $\supp(f) \subset \Fix_{n_0}(\sigma)$. Hence all these entries are zero. To see that the representations  $\pi_{x,n_0,1}$ also vanish on $f-f\delta^{n_0}$, for all $x\in\Fix_{n_0}(\sigma)$, we need only note that in that case $\pi_{x,n_0,1}(\delta^{n_0})$ is the identity operator. This established our claim.

To see that (4) implies (1), note that, by Proposition~\ref{prop:commutant_description}, (4) implies that $C(X) = C(X)^\prime$, and thus (1) follows from Theorem~\ref{thm:commutant_intersection_property}.
\end{proof}

The $C^*(\Sigma)$-analogue of the previous result is \cite[Theorem 5.4]{tomiyama_seoul_notes}; the corresponding result for $c_{00}(\Sigma)$ follows from \cite[Theorem~4.5]{SSdJ_Banff}, or \cite[Corollary~3.5]{SSdJ_Lund}.

\begin{theorem}\label{thm:simple_minimality} The following are equivalent:
\begin{enumerate}
\item The only closed ideals of $\ell^1(\Sigma)$ are $\{0\}$ and $\ell^1(\Sigma)$;
\item The only self-adjoint closed ideals of $\ell^1(\Sigma)$ are $\{0\}$ and $\ell^1(\Sigma)$;
\item $X$ has an infinite number of points, and $\Sigma$ is minimal.
\end{enumerate}
\end{theorem}

\begin{proof}
Evidently (1) implies (2). Assuming (2), we will show that (3) holds, and we start by showing that $\Sigma$ must be minimal.  Indeed, if not, then there is a point $x_0 \in X$ such that $\overline{\mathcal{O}_{\sigma} (x_0)} \neq X$. Note that $\overline{\mathcal{O}_{\sigma} (x_0)}$ is invariant under $\sigma$ and its inverse. Define
\[
I = \left\{a \in \ell^1 (\Sigma) : a_k(x)=0 \textup{ for all }k\in\mathbb Z\textup{ and all }x\in\overline{{\mathcal O}_\sigma(x_0)}\right\}.
\]
It is easy to see that $I$ is a proper non-zero self-adjoint closed ideal of $\ell^1 (\Sigma)$, contradicting (2). Hence $\Sigma$ is minimal. If $X$ is finite, then there exists a periodic point $x_0\in X$, of period $n_0\geq 1$. In that case, the associated finite dimensional unital continuous $^*$-representation $\pi_{x_0,n_0,1}$ from Section~\ref{subsec:representations} has as its kernel a proper self-adjoint closed ideal of $\ell^1(\Sigma)$, which cannot be the zero ideal for reasons of dimension. Again, this contradicts the assumption (2), so that $X$ must be infinite. Thus (2) implies (3).

In order to show that (3) implies (1), suppose that $X$ is infinite, and that $\Sigma$ is minimal. Let $I$ be a non-zero closed ideal. We claim that $X = \Aper(\sigma)$. Indeed, if $x_0\in\Per(\sigma)$, then the minimality of $\Sigma$ yields $X=\overline{\mathcal O_\sigma(x_0)}=\mathcal O_\sigma(x_0)$, contradicting the fact that $X$ is infinite. Hence $X=\Aper(\sigma)$ as claimed, so $\Sigma$ is certainly topologically free. From Theorem~\ref{thm:main_four_equivalences} it then follows that $I \cap C(X)\neq \{0\}$. It is not difficult to see that $I \cap C(X)$ is a closed ideal of $C(X)$ that is invariant under $\alpha$ and its inverse. Hence there exists a closed subset $S$ of $X$, invariant under $\sigma$ and its inverse, such that $I\cap C(X)=\{f\in C(X) : f(x)=0\textup{ for all }x\in S\}$. Since $I\cap C(X)\neq \{0\}$, we conclude that $S\neq X$, so that the minimality of $\Sigma$ implies that $S=\emptyset$. Therefore $I\cap C(X)=C(X)$, and this implies that $I=\ell^1(\Sigma)$, as was to be proved.
\end{proof}

The previous result is analogous to \cite[Theorem 5.3]{tomiyama_seoul_notes}, \cite[Theorem VIII 3.9]{davidson} and the main result in \cite{power}. The corresponding result for $c_{00}(\Sigma)$ follows from \cite[Theorem~5.1]{SSdJ_Lund}.

The greater complexity of the algebras $\ell^1(\Sigma)$, as compared to that of its $C^*$-envelope $C^*(\Sigma)$, becomes evident from the possible existence of closed ideals which are not self-adjoint. In fact, we can describe precisely when such ideals exist. In doing so, we will make use of the following:

\begin{theorem}\label{thm:rudin_non_self_adjoint_ideals}
The convolution algebra $\ell^1(\mathbb Z)$ has a non-self-adjoint closed ideal.
\end{theorem}

In fact, for every non-compact locally compact abelian group $G$, there exists a closed ideal of $L_1(G)$ which is not self-adjoint. For this rather deep result we refer to \cite[Theorem 7.7.1]{rudin_fourier_analysis}.

Now we can establish the following.

\begin{theorem}\label{thm:free_self_adjoint} The following are equivalent:
\begin{enumerate}
\item Every closed ideal of $\ell^1 (\Sigma)$ is self-adjoint;
\item $\Sigma$ is free.
\end{enumerate}
\end{theorem}

\begin{proof}
We start by showing that (2) implies (1). Suppose $\Sigma$ is free and let $I \subset \ell^1 (\Sigma)$ be a closed ideal. In order to show that it is self-adjoint, we may assume that it is non-zero and proper. Now since $\Sigma$ is, in particular, topologically free, Theorem~\ref{thm:main_four_equivalences} implies that $I \cap C(X)$ is a non-zero closed ideal of $C(X)$, and it is easy to see that it is invariant under $\alpha$ and its inverse. Hence $I \cap C(X) = \{f\in C(X) : f_{\upharpoonright X_{\pi}}=0\}$, for some closed subset $X_{\pi}$ of $X$ that is invariant under $\sigma$ and its inverse. Since $I$ is proper, $X_\pi$ is not empty. Denote by $\sigma_{\pi}$ the restriction of $\sigma$ to $X_{\pi}$, and write $\Sigma_{\pi} = (X_{\pi}, \sigma_{\pi})$.

Let $\pi : \ell^1 (\Sigma) \rightarrow \ell^1 (\Sigma) / I$ be the quotient map. We will show that $\pi$ can be factored in a certain way. To this end, denote by $\phi : \ell^1 (\Sigma) \rightarrow \ell^1 (\Sigma_{\pi})$ the $^*$-homomorphism defined by $\sum_{k\in\mathbb Z} f_k \delta^k \mapsto \sum_{k\in\mathbb Z} {f_k}_{\upharpoonright X_{\pi}} \delta_{\pi}^k$. By Tietze's extension theorem every function in $C(X_{\pi})$ can be extended to a function in $C(X)$, and an easy application of Urysohn's lemma shows that one can choose an extension whose norm is arbitrarily close to the norm of the function one extends. Using this, it is not difficult to show that the map $\Psi : \ell^1 (\Sigma_{\pi}) \rightarrow \ell^1 (\Sigma) / I$, defined by $\sum_{k\in\mathbb Z} {f_k} \delta_{\pi}^k \mapsto \sum_{k\in\mathbb Z} \widetilde{f_k} \delta^k + I$, where the $\widetilde{f_k}$ are such that $\sum_{k\in\mathbb Z} \widetilde{f_k} \delta^k \in \ell^1 (\Sigma)$ and $\widetilde{f_k}_{\upharpoonright X_{\pi}} = f_k$, is a well defined contractive homomorphism. We note that $\pi = \Psi \circ \phi$. Since $\ker(\Psi)$ is a closed ideal of $\ell^1 (\Sigma_{\pi})$,  $\ker(\Psi) \cap C(X_{\pi}) = \{0\}$ by construction, and $\Sigma_{\pi}$ is free, hence topologically free, it follows from Theorem~\ref{thm:main_four_equivalences} that $\Psi$ is injective. Thus $I = \ker(\pi) = \ker(\phi)$, and the latter is self-adjoint since $\phi$ is a $^*$-homomorphism.

We will now establish that (1) implies (2). Suppose that (1) holds, but that there exists $x_0 \in\Per_p(\sigma)$, for some $p\geq 1$. We will show that $\ell^1(\Sigma)$ has a non-self-adjoint closed ideal, and this contradiction establishes that (1) implies (2). This ideal is found by constructing a continuous surjective $^*$-homomorphism $\Psi:\ell^1(\Sigma)\to A$, where $A$ is an involutive algebra, which is also a normed linear space, and which has a non-self-adjoint closed ideal $I$. In that case, $\Psi^{-1}(I)$ is a closed ideal of $\ell^1(\Sigma)$ and it cannot be self-adjoint, since the surjectivity of $\Psi$ would then imply that $I$ is self-adjoint. For $A$ we take $M_p(\AC(\mathbb{T}))$, the algebra of $p\times p$-matrices with coefficients in the Banach $^*$-algebra of continuous functions on $\mathbb T$ with absolutely convergent Fourier series. The involution $M_p(\AC(\mathbb{T}))$ is the canonical one, and as norm we take $\|A\| = \max_{1 \leq i, j \leq p} \|f_{ij}\|_{\AC(\mathbb{T})}$, for $A \in M_p(\AC(\mathbb{T}))$. Here $\Vert\,.\,\Vert_{\AC({\mathbb T})}$ is the usual norm on $\AC(\mathbb{T})$, i.e., the sum of the absolute values of the Fourier coefficients of the function involved. As a Banach $^*$-algebra, $\AC(\mathbb{T})$ is, via Fourier transform, isometrically $^*$-isomorphic with $\ell^1(\mathbb{Z})$. Therefore, Theorem~\ref{thm:rudin_non_self_adjoint_ideals} shows that $\AC (\mathbb{T})$ has a non-self-adjoint closed ideal $I$, and hence $M_p(\AC(\mathbb{T}))$ has a non-self-adjoint closed ideal, namely $M_p(I)$. Hence it remains to construct a continuous surjective $^*$-homomorphism $\Psi:\ell^1(\Sigma)\to M_p(\AC(\mathbb{T}))$.

This is accomplished by combining, for $z \in \mathbb{T}$, the $p$-dimensional representations $\pi_{x_0,p,z}$ of $\ell^1 (\Sigma)$, which are associated with $x_0\in\Fix_p(\sigma)$ as described in Section~\ref{subsec:representations}.
For $f\in C(X)$ we define
\[
\Psi(f)=\left(
\begin{array}{cccc}
f(x_0) & 0 & \ldots & 0 \\
0 & f (\sigma x_0) & \ldots & 0 \\
\vdots & \vdots & \ddots & \vdots \\
0 & 0 & \ldots & f(\sigma^{p-1} x_0)
\end{array} \right)\in M_p(\AC(\mathbb{T})),
\]
and we let
\[
\Psi(\delta)=\left(
\begin{array}{ccccc}
0 & 0 & \ldots & 0 & z \\
1 & 0 & \ldots & 0 & 0 \\
0 & 1 & \ldots & 0&0\\
\vdots & \vdots & \ddots &\vdots &\vdots \\
0 & 0 & \ldots & 1&0
\end{array} \right)\in M_p(\AC(\mathbb{T})).
\]
Then $\Psi(\delta)$ is an invertible (in fact unitary) element of $M_p(\AC(\mathbb{T}))$, so that we can define
\[
\Psi\left(\sum_{k\in\mathbb Z}a_k\delta^k\right)=\sum_{k\in\mathbb Z}\Psi(a_k)\Psi(\delta)^k,
\]
for $\sum_{k\in\mathbb Z}a_k\delta^k\in c_{00}(\Sigma)$. Since $\Vert \Psi(\delta)\Vert=\Vert\Psi (\delta)^{-1}\Vert =1$, and $\Vert\Psi(f)\Vert\leq\Vert f\Vert$, for $f\in C(X)$,  $\Psi:c_{00}(\Sigma)\to M_p(\AC(\mathbb{T}))$ is contractive and it is easy to check that it is a unital $^*$-homomorphism. Using that $\Psi(\delta)^p=z\cdot\id$, some moments thought show that one has explicitly that
\begin{align}\label{eq:AC_representation}
\Psi \left( \sum_{k\in\mathbb Z} a_k \delta^k\right) &=  \Psi\left( \sum_{r=0}^{p-1}\sum_{l\in\mathbb Z} a_{lp+r} \delta^{lp+r}\right)\notag\\
&=\sum_{r=0}^{p-1}\sum_{l\in\mathbb Z} \Psi(a_{lp+r})\cdot z^l \cdot \Psi(\delta)^r \notag\\
&=\sum_{r=0}^{p-1}\left(\sum_{l\in\mathbb Z} \Psi(a_{lp+r}) z^l\right)\Psi(\delta)^r \\
&=\left (\begin{array}{lcl}
\sum_{l\in\mathbb Z} a_{lp} (x_0) z^l &  \ldots & \sum_{l\in\mathbb Z} a_{(l-1)p+1} (x_0) z^l \notag\\
\sum_{l\in\mathbb Z} a_{lp+1} (\sigma x_0) z^l   &  \ldots & \sum_{l\in\mathbb Z} a_{(l-1)p+2} (\sigma x_0) z^l \\
\sum_{l\in\mathbb Z} a_{lp+2} (\sigma^2 x_0) z^l  &  \ldots & \sum_{l\in\mathbb Z} a_{(l-1)p+3}(\sigma^2 x_0) z^l \\
\quad\quad\quad\quad\vdots & \vdots &\quad\quad\quad\quad\vdots\\
\sum_{l\in\mathbb Z} a_{lp+(p-1)} (\sigma^{p-1} x_0) z^l  &  \ldots & \sum_{l\in\mathbb Z} a_{lp}(\sigma^{p-1} x_0) z^l\\
\end{array}
\right)\notag,
\end{align}
for $\sum_{k\in\mathbb Z}a_k\delta^k\in c_{00}(\Sigma)$. By density of $c_{00}(\Sigma)$, $\Psi$ extends to a unital contractive $^*$-homomorphism $\Psi:\ell^1(\Sigma)\to M_p(\AC(\mathbb{T}))$, which is explicitly given by \eqref{eq:AC_representation} again, for $\sum_{k\in\mathbb Z}a_k\delta^k\in\ell^1(\Sigma)$. We claim that $\Psi$ is surjective.
Indeed, since each complex number $a_i (\sigma^j x_0)$, for $i\in\mathbb Z$, and $j=0,\ldots,p-1$, occurs precisely once (somewhere in row $j+1$) as a Fourier coefficient in the matrix of $\Psi\left(\sum_{k\in\mathbb Z} a_k \delta^k\right)$ in \eqref{eq:AC_representation}, one sees that finding a pre-image $\sum_{k\in\mathbb Z} a_k \delta^k$ of a given element of $M_p(\AC(\mathbb{T}))$ amounts to prescribing the numbers $a_k (\sigma^j x_0)$, for $k\in\mathbb Z$, and $j=0,\ldots,p-1$. Since the points $x_0,\ldots,\sigma^{p-1}(x_0)$ are all different, Urysohn's lemma implies that, for $k\in\mathbb Z$, one can actually find $a_k\in C(X)$ such that, for $j=1,\ldots.p-1$, $a_k(\sigma^j)$ has the prescribed value, and, moreover, such that $\|a_k\|_\infty=\max_{j=0,\ldots,p-1}|a_k(\sigma^j x_0)|$. With these $a_k$, we have $\sum_{k\in\mathbb Z}\|a_k\|_\infty\leq\sum_{k\in\mathbb Z}\sum_{j=0}^{p-1}|a_k(\sigma^j x_0)|$. Since the double series will then converge as a consequence of our starting out with an element of $M_p(\AC(\mathbb{T}))$, the tentative pre-image $\sum_{k\in\mathbb Z} a_k\delta^k$ as constructed is thus seen to be indeed in $\ell^1(\Sigma)$. We conclude that $\Psi$ is surjective, and hence $\Psi:\ell^1(\Sigma)\to M_p(\AC(\mathbb{T}))$ is a continuous surjective $^*$-homomorphism, as desired.
\end{proof}

If $X$ consists of one finite orbit of $p$ elements, then $C^*(\Sigma)$ is isomorphic to $M_p(C(\mathbb T))$ \cite[Proposition~3.5]{tomiyama_rev_math_phys}. Since $C(\mathbb T)$ is the $C^*$-envelope of $AC(\mathbb T)$, the following structure theorem for $\ell^1(\Sigma)$ is therefore quite natural. It is a direct consequence of the second part of the above proof, since, under the hypothesis that $X$ consists of one finite orbit, the map $\Psi$ figuring in that proof is then evidently also injective. If $X$ consists of one point, the result reduces to the standard isomorphism between $\ell^1(\mathbb Z)$ and $\AC(\mathbb T)$. 

\begin{theorem}\label{thm:structure_theorem_for_one_orbit}
Suppose that $X=\{x_0,\sigma x_0,\ldots,\sigma^{p-1}x_0\}$ consists of one finite orbit of $p$ elements, for some $x_0\in X$, and $p\geq 1$. Then the map $\Psi:\ell^{1}(\Sigma)\to M_p(\AC (\mathbb T))$ in \eqref{eq:AC_representation} is a $^*$-isomorphism between the involutive algebras $\ell^{1}(\Sigma)$ and $M_p(\AC (\mathbb T))$, and a linear homeomorphism of Banach spaces.
\end{theorem}

Theorem~\ref{thm:structure_theorem_for_one_orbit} will be used in determining when $\ell^1(\Sigma)$ is ($^*$-)prime in Theorem~\ref{thm:prime_transitive}. It also has the following consequence on the Hermitian nature of $\ell^1(\Sigma)$ for finite $X$. It is an open question whether $\ell^1(\Sigma)$ is always Hermitian, or, if not, which conditions on the dynamics are equivalent with this property.

\begin{theorem}\label{thm:finite_hermitian}
If $X$ consists of a finite number of points, then $\ell^1(\Sigma)$ is Hermitian.
\end{theorem}

\begin{proof}
Using a direct sum decomposition of $\ell^1(\Sigma)$ corresponding to the orbits in $X$, one sees that we may assume that $X$ consists of one finite orbit of $p\geq 1$ elements. In that case, Theorem~\ref{thm:structure_theorem_for_one_orbit} implies that we may just as well show that $M_p(\AC(\mathbb{T}))$ is Hermitian. Hence, let $a$ be a self-adjoint element of $M_p(\AC(\mathbb{T}))$; we must show that its spectrum in $M_p(\AC(\mathbb{T}))$ is real. Now, for every $ z\in\mathbb T$, there is a natural $^*$-homomorphism $\textup{ev}_z:M_p(\AC (\mathbb T))\to\mathbb C$, corresponding to evaluation at $z$, and we claim that, for still arbitrary $a\in M_p(\AC(\mathbb{T}))$, $a$ is invertible in $M_p(\AC(\mathbb{T}))$ precisely when $\textup{ev}_z(a)$ is invertible in $M_p(\mathbb C)$, for all $z\in\mathbb T$. Indeed, invertibility of $a$ in $M_p(\AC(\mathbb{T}))$ surely implies invertibility of $\textup{ev}_z(a)$ in $M_p(\mathbb C)$. As to the converse: if the determinant of $\textup{ev}_z(a)$ is non-zero, for every $z\in\mathbb T$, then $z\mapsto\det(\textup{ev}_z(a))$ is an element of $\AC(\mathbb{T})$ having no zero on $\mathbb T$. By Wiener's classical result, its reciprocal is in $\AC(\mathbb{T})$ again, and this shows that the pointwise inverses of $\textup{ev}_z(a)$ in $M_p(\mathbb C)$, for $z\in\mathbb T$, combine to the inverse of $a$ in $M_p(\AC(\mathbb{T}))$. This establishes our claim. We conclude that, for arbitrary $a\in M_p(\AC(\mathbb{T}))$, the spectrum of $a$ in $M_p(\AC(\mathbb{T}))$ is the union of the spectra of the $\textup{ev}_z(a)$ in $M_p(\mathbb C)$, as $z$ ranges over $\mathbb T$. If $a$ is self-adjoint, then so are the $\textup{ev}_z(a)$, for all $z\in\mathbb T$. Hence these spectra are all real and the same therefore holds for the spectrum of $a$.
\end{proof}

\begin{remark}\label{rem:tensorproduct}
It is an open question when a tensor product of two unital Hermitian Banach algebras is Hermitian again. It is known to hold when at least one of the factors is commutative \cite[Theorem~34.15]{fragoulopoulou}. Our proof of Theorem~\ref{thm:finite_hermitian}, where the spectrum of a self-adjoint element is seen to be real because it is the union of spectra which are known to be real, bears some resemblance to the proof of this fact in \cite[Theorem~34.15 and Theorem~31.20.(2)]{fragoulopoulou}.
\end{remark}

We conclude this section by determining when $\ell^1(\Sigma)$ is ($^*$-)prime in Theorem~\ref{thm:prime_transitive}, using Theorem~\ref{thm:structure_theorem_for_one_orbit} as an ingredient. As a preparation, we need the following two topological lemmas.

\begin{lemma}\label{lem:not_transitive_open_sets} The following are equivalent:
\begin{enumerate}
\item There exist two disjoint non-empty open subsets $O_1$ and $O_2$ of $X$, both invariant under $\sigma$ and its inverse, such that $\overline{O_1} \cup \overline{O_2} = X$;
\item $\Sigma$ is not topologically transitive.
\end{enumerate}
\end{lemma}

\begin{proof}
The first statement evidently implies the second. As to the converse, if the system is not topologically transitive, then there exist non-empty open subsets $U,V$ of $X$ such that $\sigma^n (U) \cap V = \emptyset$, for all $n\in\mathbb Z$. Therefore $O_1 = \bigcup_{n \in \mathbb{Z}} \sigma^n (U)$ is a non-empty open set invariant under $\sigma$ and $\sigma^{-1}$, and $O_1\cap V=\emptyset$. As a consequence, $\overline{O_1}$ is a closed set invariant under $\sigma$ and $\sigma^{-1}$, and $\overline{O_1}\cap V=\emptyset$. It follows that $O_2 = {\overline{O_1}}^c$ is an open set, invariant under $\sigma$ and $\sigma^{-1}$, containing $V$, and hence non-empty. Since we even have $\overline{O_1} \cup O_2 = X$, the result follows.
\end{proof}

\begin{lemma}\label{lem:transitive_single_orbit}
If $\Sigma$ is topologically transitive, and there exists $n_0 \geq  1$ such that $X = \Fix_{n_0} (\sigma)$, then $X$ consists of a single orbit and is thus finite.
\end{lemma}

\begin{proof}
Fix $x\in X$ and assume that some $y\in X$ is not in the orbit of $x$. Then there exist an open neighbourhood $V_y$ of $y$, and an open neighbourhood $V_k$ of $\sigma^k x$, for $k=0,\ldots,n_0-1$, such that $V_k\cap V_y=\emptyset$, for $k=0,\ldots,n_0-1$. Define the open neighbourhood
\[
U_x = V_0 \cap \sigma^{-1} (V_1) \cap \sigma^{-2} (V_2) \cap \ldots \cap \sigma^{-(n_0 -1)} (V_{n_0-1})
\]
of $x$, and consider the non-empty open subset $W_x  = \cup_{i=0}^{n_0 -1} \sigma^i (U_x)$ of $X$. If $z\in W_x\cap V_y$, say $z\in\sigma^{i_0}(U_x)\cap V_y$, for some $0\leq i_0\leq n_0-1$, then
\[
z\in\sigma^{i_0}(V_0) \cap \sigma^{i_0-1} (V_1) \cap \sigma^{i_0-2} (V_2) \cap \ldots \cap \sigma^{i_0-(n_0 -1)} (V_{n_0-1})\cap V_y\subset V_{i_0}\cap V_y,
\]
contradicting that $V_{i_0}\cap V_y=\emptyset$. Hence $W_x\cap V_y=\emptyset$. Furthermore, since $\sigma^{n_0}=\id_X$, $W_x$ is invariant under $\sigma$ and $\sigma^{-1}$, so that $\sigma^n(W_x)\cap V_y=W_x\cap V_y=\emptyset$, for all $n\in\mathbb Z$. This contradicts that $\Sigma$ is topologically free, and hence $X$ must coincide with the orbit of $x$.
\end{proof}

\begin{theorem}\label{thm:prime_transitive} The following are equivalent:
\begin{enumerate}
\item If $I_1$ and $I_2$ are two non-zero closed ideals of $\ell^1(\Sigma)$, then $I_1\cap I_2\neq \{0\}$.
\item If $I_1$ and $I_2$ are two non-zero self-adjoint closed ideals of $\ell^1(\Sigma)$, then $I_1\cap I_2\neq \{0\}$.
\item $X$ has an infinite number of points, and $\Sigma$ is topologically transitive.
\end{enumerate}
\end{theorem}

\begin{proof}
Evidently (1) implies (2). Assuming (2), we will show that (3) holds, and we start by showing that $\Sigma$ must be topologically transitive. Indeed, if not, then, by Lemma~\ref{lem:not_transitive_open_sets}, there exist two disjoint non-empty open sets $O_1$ and $O_2$, both invariant under $\sigma$ and $\sigma^{-1}$, and such that $\overline{O_1} \cup \overline{O_2} = X$. For $i=1,2$, let
\[
I_i = \left\{\sum_{k\in\mathbb Z} a_k \delta^k \in \ell^1 (\Sigma) : f_k \in \ker(\overline{O_i}) \textup{ for all } n\right\},
\]
where, for $S\subset X$, $\ker (S)=\{f\in C(X) : f\upharpoonright_S=0\}$. Then $I_1$ and $I_2$ are non-zero self-adjoint closed ideals of $\ell^1 (\Sigma)$. Since $E(I_i) = \ker(\overline{O_i})$, for $i=1,2$, we have
\[
E(I_1 \cap I_2) \subset E(I_1) \cap E(I_2) = \ker(\overline{O_1}) \cap \ker(\overline{O_2})= \ker(\overline{O_1} \cup \overline{O_2})= \ker(X) = \{0\},
\]
from which it follows that $I_1 \cap I_2 = \{0\}$. This contradicts assumption (2), and hence $\Sigma$ must be topologically transitive. Having established this, suppose that $X$ is finite. As a consequence of Lemma~\ref{lem:transitive_single_orbit}, $X$ then consists of one finite orbit of, say, $p\geq 1$ elements. Consequently, by Theorem~\ref{thm:structure_theorem_for_one_orbit}, $\ell^1(\Sigma)$ is isomorphic, as Banach space and as involutive algebra, with $M_p(\AC (\mathbb T))$. This is a contradiction, since $M_p(\AC (\mathbb T))$ has two non-zero self-adjoint closed ideals with zero intersection. In order to see this, consider, for every $ z\in\mathbb T$, the natural continuous $^*$-homomorphism $\textup{ev}_z:M_p(\AC (\mathbb T))\to\mathbb C$, corresponding to evaluation at $z$. Choose two non-empty proper closed subsets $C_1$ and $C_2$ of $\mathbb T$ such that $C_1\cup C_2=\mathbb T$, and define $I_i=\bigcap_{z\in C_i}\ker(\textup{ev}_z)$, for $i=1,2$. Then clearly $I_1$ and $I_2$ are self-adjoint closed ideals of $M_p(\AC (\mathbb T))$, and $I_1\cap I_2=\{0\}$. However, $I_i\neq\{0\}$, for $i=1,2$. Indeed, since, by \cite[Corollary~7.2.3]{larsen}, $\ell^1(\mathbb Z)$ is a regular Banach algebra, by \cite[Theorem~7.1.2]{larsen} there exists, for $i=1,2$, a non-zero $f_i\in\AC(\mathbb T)$ which vanishes at $C_i$. The corresponding diagonal matrices, for example, then show that indeed $I_i\neq\{0\}$, for $i=1,2$. This establishes our claim about $M_p(\AC (\mathbb T))$, and hence the first statement implies the second.

In order to show that (3) implies (1), suppose that $X$ is infinite and that $\Sigma$ is topologically transitive. To start with, we claim that $\Sigma$ is topologically free. If not, then, by the first part of Lemma~\ref{lem:topological}, there is an integer $n_0 \geq 1$ such that $\Fix_{n_0} (\sigma)$ has non-empty interior. Since $\Fix_{n_0} (\sigma)$ is invariant under $\sigma$ and $\sigma^{-1}$, we conclude that $\Fix_{n_0} (\sigma)^\circ$ is a non-empty open subset of $X$, invariant under $\sigma$ and $\sigma^{-1}$. Moreover, we cannot have $\Fix_{n_0}(\sigma)=X$, since then Lemma~\ref{lem:transitive_single_orbit} would imply that $X$ is finite. Hence $\Fix_{n_0}(\sigma)^c$ is a non-empty open subset of $X$. Since then $\sigma^n(\Fix_{n_0}(\sigma)^\circ)\cap \Fix_{n_0}(\sigma)^c=\Fix_{n_0}(\sigma)^\circ\cap \Fix_{n_0}(\sigma)^c=\emptyset$, for all $n\in\mathbb Z$, this would contradict the topological transitivity of $\Sigma$.
Hence $\Sigma$ is topologically free, as claimed. Now let $I_1$ and $I_2$ be two non-zero closed ideals of $\ell^1 (\Sigma)$, and suppose that $I_1 \cap I_2 = \{0\}$. Then $I_1$ and $I_2$ are both proper ideals. Hence $I_1 \cap C(X)$ and $I_2 \cap C(X)$ are proper closed ideals of $C(X)$, obviously with zero intersection, and invariant under $\alpha$ and its inverse. Topologically freeness of $\Sigma$ implies that $I_1 \cap C(X)$ and $I_2 \cap C(X)$ are both non-zero, by Theorem~\ref{thm:main_four_equivalences}. We conclude that there exist proper non-empty closed subsets $C_1, C_2$ of $X$, invariant under $\sigma$ and its inverse, such that $I_1 \cap C(X) = \ker(C_1)$, and $I_2 \cap C(X) = \ker(C_2)$, where, for $S\subset X$, $\ker (S)=\{f\in C(X) : f\upharpoonright_S=0\}$. We note that $\{0\} = I_1 \cap I_2 \cap C(X) = \ker(C_1) \cap \ker (C_2) = \ker(C_1 \cup C_2)$, whence $C_1 \cup C_2 = X$. This implies, using that $C_2$ is proper and closed, that $C_1^\circ \supset C_2^c \neq \emptyset$. Hence $C_1^\circ$ is a non-empty open subset of $X$, invariant under $\sigma$ and its inverse. Furthermore, since $C_1$ is proper and closed, $C_1^c$ is open and non-empty. Hence $\sigma^n (C_1^\circ) \cap C_1^c = C_1^\circ\cap C_1^c=\emptyset$, for all $n\in\mathbb Z$. This contradicts the topological transitivity of $\Sigma$, and we conclude that we must have $I_1 \cap I_2 \neq \{0\}$. Hence (1) holds.
\end{proof}

The $C^*$-analogue of the previous result is \cite[Theorem 5.5]{tomiyama_seoul_notes}; the corresponding result for $c_{00}(\Sigma)$ follows from \cite[Theorem~7.6]{SSdJ_Lund}.

\subsection*{Acknowledgements}

This work was supported by a visitor's grant of the Netherlands Organisation for Scientific Research (NWO).


\begin{thebibliography}{99}

\bibitem{davidson} Davidson, K.R., \emph{$C^*$-algebras by example}, Fields Institute Monographs, No.\ 6, Amer.\ Math.\ Soc., Providence, RI, 1996.

\bibitem{fragoulopoulou} Fragoulopoulou, M., \emph{Topological algebras with involution}, North-Holland Mathematics Studies, {\bf 200}, Elsevier Science B.V., Amsterdam, 2005.

\bibitem{kaniuth} Kaniuth, E., \emph{A course in commutative Banach algebras}, Springer, New York, 2009.

\bibitem{kelley} Kelly, J.L., \emph{General topology}, Springer-Verlag, New York-Heidelberg-Berlin, 1975.

\bibitem{larsen} Larsen, R., \emph{Banach algebras. An introduction}. Pure and Applied Mathematics, No.\ 24, Marcel Dekker, Inc., New York, 1973.

\bibitem{palmer} Palmer, Th.W., \emph{Banach algebras and the general theory of $^*$-algebras. Vol.\ 2. $^*$-algebras}, Encyclopedia of Mathematics and its Applications, {\bf 79}, Cambridge University Press, Cambridge, 2001.

\bibitem{power} Power, S.C., \emph{Simplicity of $C^*$-algebras of minimal dynamical systems}, J.\ London Math.\ Soc.\ \textbf{18} (1978), 534-538.

\bibitem{rudin_fourier_analysis} Rudin, W., \emph{Fourier analysis on groups}, Interscience Publishers, New York, London, 1962.

\bibitem{rudin_functional_analysis} Rudin, W., \emph{Functional analysis (2nd Ed.)}, McGraw-Hill, Singapore, 1991.

\bibitem{sakai} Sakai, S., \emph{$C^*$-algebras and $W^*$-algebras}, Springer, Berlin-Heidelberg, 1998.

\bibitem{SSdJ_IJM} Svensson, C., Silvestrov, S., de Jeu, M., \emph{Dynamical systems and commutants in crossed products}, Internat.\ J.\ Math.\ \textbf{18} (2007), 455-471.

\bibitem{SSdJ_Banff} Svensson, C., Silvestrov, S., de Jeu, M., \emph{Dynamical systems associated with crossed products}, Acta Appl.\ Math.\ \textbf{108} (2009), 547-559.

\bibitem{SSdJ_Lund}  Svensson, C., Silvestrov, S., de Jeu, M., \emph{Connections between dynamical systems and crossed products of Banach algebras by $\mathbb{Z}$}, in ``Methods of Spectral Analysis in Mathematical Physics, Conference on Operator theory, Analysis and Mathematical Physics (OTAMP) 2006, Lund, Sweden'', Janas, J., Kurasov, P., Laptev, A., Naboko, S., Stolz, G.\ (Eds.), Operator Theory: Advances and Applications \textbf{186}, Birkh\"auser, Basel, 2009, 391-401.

\bibitem{ST_JFA} Svensson, C., Tomiyama, J., \emph{On the commutant of $C(X)$ in $C^*$-crossed products by $\mathbb{Z}$ and their representations}, J.\ Funct.\ Anal.\ \textbf{256} (2009), 2367-2386.

\bibitem{tomiyama_book} Tomiyama, J., \emph{Invitation to $C^*$-algebras and topological dynamics}, World Sci., Singapore, New Jersey, Hong Kong, 1987.

\bibitem{tomiyama_seoul_notes}  Tomiyama, J., \emph{The interplay between topological dynamics and theory of $C^*$-algebras}, Lecture Note no.2, Global Anal.\ Research Center, Seoul, 1992.

\bibitem{tomiyama_rev_math_phys} Tomiyama, J., \emph{$C^*$-algebras and topological dynamical systems}, Rev.\ Math.\ Phys.\ {\bf 8} (1996), 741-760.

\bibitem{tomiyama_kyoto_notes} Tomiyama, J., \emph{The interplay between topological dynamics and theory of $C^*$-algebras. II}, Kyoto, 2000.

\bibitem{williams} Williams, D.P., \emph{Crossed products of $C{^*}$-algebras}, Mathematical Surveys and Monographs, No.\ 134, American Mathematical Society, Providence, RI, 2007.

\end{thebibliography}
\end{document}